\documentclass[11pt,reqno]{amsart}

\usepackage[latin1]{inputenc}
\usepackage[T1]{fontenc}
\usepackage{lmodern}
\usepackage[english]{babel}
\usepackage{stmaryrd}
\usepackage{tikz}
\usepackage{bbm}
\usepackage{amsmath,amssymb,amsfonts,amsthm}
\usepackage{mathtools,accents}
\usepackage{mathrsfs}
\usepackage{xfrac}
\usepackage{array} 
\usepackage{aliascnt}
\usepackage{booktabs} 
\usepackage{array} 

\usepackage{verbatim} 
\usepackage{subfig} 
\usepackage{mathrsfs}
\usepackage{mathrsfs, dsfont}
\usepackage{amssymb}
\usepackage{amsthm}
\usepackage{amsmath,amsfonts,amssymb,esint}
\usepackage{graphics,color}
\usepackage{enumerate}
\usepackage{mathtools,centernot}

\usepackage{microtype}
\usepackage{paralist} 
\usepackage{cases}
\usepackage[initials]{amsrefs}
\allowdisplaybreaks

\usepackage{braket}
\usepackage{bm}

\usepackage[citecolor=blue,colorlinks]{hyperref}
\addto\extrasenglish{}

\usepackage{enumerate}
\usepackage{xcolor}

\usepackage{aliascnt}

\makeatletter
\def\newaliasedtheorem#1[#2]#3{
  \newaliascnt{#1@alt}{#2}
  \newtheorem{#1}[#1@alt]{#3}
  \expandafter\newcommand\csname #1@altname\endcsname{#3}
}
\makeatother

\theoremstyle{plain}
\newtheorem{theorem}{Theorem}[section]
\newaliasedtheorem{lemma}[theorem]{Lemma}
\newaliasedtheorem{prop}[theorem]{Proposition}
\newaliasedtheorem{claim}[theorem]{Claim}
\newaliasedtheorem{corollary}[theorem]{Corollary}

\theoremstyle{remark}

\theoremstyle{definition}
\newaliasedtheorem{definition}[theorem]{Definition}
\newaliasedtheorem{example}[theorem]{Example}

\theoremstyle{remark}
\newaliasedtheorem{remark}[theorem]{Remark}
\newaliasedtheorem{ex}[theorem]{Example}

\numberwithin{equation}{section}

\def\eps{\varepsilon}

\def\R{\mathbb R}
\def\N{{\mathbb N}}

\DeclareMathOperator{\dv}{div}

\DeclareMathOperator{\Ker}{Ker}

\DeclareMathOperator{\rank}{rank}
\DeclareMathOperator{\loc}{loc}

\DeclareMathOperator{\Sym}{Sym}

\DeclareMathOperator{\tr}{tr}
\DeclareMathOperator{\id}{id}
\DeclareMathOperator{\diag}{diag}
\DeclareMathOperator{\tra}{trail}

\begin{document}

\title[Improved regularity of second derivatives for subharmonic functions]{Improved regularity of second derivatives for subharmonic functions}

\author{Xavier Fern\'andez-Real}
\address{EPFL SB, Station 8, CH-1015 Lausanne, Switzerland}
\email{xavier.fernandez-real@epfl.ch}

\author{Riccardo Tione}
\address{EPFL SB, Station 8, CH-1015 Lausanne, Switzerland}
\email{riccardo.tione@epfl.ch}


\begin{abstract}
In this note, we prove that if a subharmonic function $\Delta u\ge 0$ has pure second derivatives $\partial_{ii} u$ that are signed measures, then their negative part $(\partial_{ii} u)_-$ belongs to $L^1$ (in particular, it is not singular). We then show that this improvement of regularity cannot be upgraded to $L^p$ for any $p > 1$. We finally relate this problem to a natural question on the one-sided regularity of solutions to the obstacle problem with rough obstacles. 
\end{abstract}

\maketitle

\section{Introduction}

Let us consider the following vague question:
\[
\text{If $u$ is subharmonic, $\Delta u\ge 0$, can it happen that $\partial_{ii} u = -\delta_0$?}
\]

 More generally, we consider the following problem, where we denote by $\mathcal{M}$ the space of locally finite (signed) Radon measures.  
\\[0.5cm]
\noindent {\bf Problem:} Let $u$ be subharmonic in the distributional sense, and let us assume that $D^2 u \in \mathcal{M}$. Is it then true that  $(D^2 u)_-\in L^p_{\rm loc}$ for some $p \ge 1$ (in particular, it has no singular part)?
\vspace{0.5cm}

In this note we discuss the validity of the previous statement. In particular, we show that if second derivatives of a subharmonic function $u$ are Radon measures, then their negative part is in $L^1$. We then provide counterexamples to show that, in general, it is not in $L^p$ for any $p > 1$.

The problem is motivated by a question in free boundary problems, on the one-sided regularity of solutions to the obstacle problem with rough obstacles. See Section~\ref{sec.obst} for a contextualization in that setting.  

Our main results can be summarized by the following statement:

\begin{theorem}
\label{thm.main}
Let $\Omega \subset \R^n$ be open, and let $u \in L^1_{\loc}(\Omega)$ be subharmonic, i.e. $\Delta u \ge 0$ in the sense of distributions. If $\partial_{ii}u$ are (locally) finite signed measures for all $1 \le i \le n$, then their singular part (with respect to the Lebesgue measure) is a positive measure, for all $i$. In particular,
\[
(\partial_{ii}u)_- \in L^1_{\loc}(\Omega).
\]
Moreover, this result is sharp, in the sense that there exists a subharmonic and Lipschitz $u: \Omega \to \R$ such that $u\in W^{2,1}(\Omega)$, but $(\partial_{ii}u)_{-} \notin L^p$, for any $1\le i \le n$, and any $p > 1$. 
\end{theorem}

\subsection{Notation}

In the following, $\Omega\subset \R^n$ is always an open domain, and $\mathcal{M}(\Omega)$ denotes the space of signed Radon measures on $\Omega$. Similarly, we denote by $\mathcal{M}_+(\Omega)$ the space of (nonnegative) Radon measures on $\Omega$.  In this note, we say that $\mu \in \mathcal{M}(\Omega)$ is singular if it is singular with respect to the Lebesgue measure. In particular, for any vector-valued measure $\nu \in (\mathcal{M}(\Omega))^N$, we can consider its unique Lebesgue decomposition, see \cite[Theorem 1.31]{EVG}:
\[
\nu = Fdx + Ad\nu^s,
\]
where $F \in L^1(\Omega,\R^N)$, $\nu^s\in\mathcal{M}_+(\Omega)$ is a finite, singular measure, and $A$ is a Borel vector field of $\R^N$, taking values in the sphere $\mathbb{S}^{N - 1}$ a.e..

We denote by $f_+ \ge 0$ the positive part of a function $f$, and by $M_+$ the positive part of the symmetric matrix $M\in \Sym(n)$. Namely, if $M = O D O^T$ for some diagonal matrix $D$ and orthogonal matrix $O$, then $M_+ = OD_+ O^T$, where $D_+$ is the entrywise positive part of $D$. Analogously, let $f_{-} \ge 0$ and $M_{-}$ the negative part of $f$ and $M$, respectively. We have $f = f_+ - f_-$ and $M = M_+-M_-$.

For a measure $\mu\in\mathcal{M}(\Omega)$, we can decompose uniquely (up to $\mu$-negligible sets) $\mu = \mu_+ - \mu_{-}$, with $\mu_+ \ge 0$ and $\mu_{-} \ge 0$. Finally, for a symmetric matrix-valued signed measure $\nu\in (\mathcal{M}(\Omega))^{n\times n}$ (where $\nu_{ij} = \nu_{ji}$ for all $1\le i, j\le n$), with decomposition $\nu = Fdx + Ad\nu^s$, we can split $\nu = \nu_+ - \nu_-$ where we have denoted
\[
\nu_\pm := F_\pm dx + A_\pm d\nu^s,
\] 
and where $F_\pm(x)$ and $A_\pm(x)$ are the positive and negative part of $F(x), A(x)\in \Sym(n)$.
\section{Improvement of regularity}

\begin{prop}\label{BV}
Let $u \in L^1_{\rm loc}(\Omega)$ be subharmonic in the distributional sense, and let $D^2 u\in ({\mathcal{M}}(\Omega))^{n\times n}$. Let us decompose  
\[
D^2 u = (D^2 u)_+ - (D^2 u)_-.
\]
 Then,  we have
\[
(D^2 u)_- \in L^1_{\rm loc}(\Omega),
\]
that is, the negative part of $D^2 u$ is not singular. 
\end{prop}
\begin{remark}
In particular, the polar part of $(D^2 u)^s$ is nonnegative and we deduce that 
\[
(\partial_{ee}u)^s \ge 0,
\]
for any $e\in \mathbb{S}^{n-1}$ if $\nu^s$ denotes the singular part of the measure $\nu$ with respect to the Lebesgue measure.
\end{remark}
\begin{proof}
By Alberti's rank-one theorem, \cite{Alb93}, if we decompose $D^2 u$ into 
\[
D^2 u = F dx + A d\mu,
\]
then $A\in L^1_{\rm loc}(\mu)$ is of rank-one at $\mu\in\mathcal{M}_+(\Omega)$ a.e. point. Moreover, since the measure $D^2u$ is symmetric (i.e. $(D^2u)_{ij} = (D^2u)_{ji}$ for all $i, j$ in the sense of measures), we deduce that $A(x)$ is a symmetric matrix at $\mu$-a.e. $x \in \Omega$. It follows that,
\[
A(x) = \lambda(x) a(x)\otimes a(x)\quad \text{at $\mu$-a.e. $x \in \Omega$},
\]
for Borel $\lambda \in L^1(\Omega;\mu)$, $a: \Omega \to \mathbb{S}^{n-1}$. By subharmonicity, we further deduce
\[
\tr(A(x)) = \lambda(x)|a(x)|^2 = \lambda(x) \ge 0
\]
at $\mu$-a.e. $x$. Hence, for any $v \in \R^n$,
\[
(A(x)v,v) = \lambda(x) (a,v)^2 \ge 0,
\]
at $\mu$-a.e. $x \in \Omega$. This concludes the proof.
\end{proof}

The previous proposition can be improved in two ways. First, we can consider more general operators than $\Delta$. In particular, consider any Borel matrix-field $A= A(x)$, with 
\[
A \in \{M: M \in C^0(\Omega,\Sym(n)), \dv(M) \in L_{\loc}^n(\Omega,\R^n)\}.
\]
If
\[
A(x) > 0\quad \text{a.e. in }\Omega \text{ and }\dv(A \,Du) \ge 0,
\]
in the sense of distributions, then it is easy to see that the same proof of the previous result yields once again $(D^2u)_{-} \in L^1_{\loc}$. Of course, Proposition \ref{BV} is a particular instance of this statement, obtained by taking $A \equiv \id_n$. The other way it can be improved is by noticing that $\Delta u$ only involves pure derivatives of $u$, and hence it is natural to ask whether requirement $D^2 u\in ({\mathcal{M}}(\Omega))^{n\times n}$ can be replaced by $\partial_{ii}u \in {\mathcal{M}}(\Omega)$, for all $1 \le i \le n$. Notice that the latter condition does not imply the former even if $\Delta u\ge 0$, as can be seen for instance through \cite[Theorem 3]{CFM}.

 The next proposition precisely tells us that the condition $\partial_{ii}u \in \mathcal{M}(\Omega)$ is enough to infer the same conclusion of Proposition \ref{BV}.

\begin{prop}\label{deri}
Let $u \in L^1_{\loc}(\Omega)$ be subharmonic in the distributional sense, such that $\partial_{ii} u \in \mathcal{M}(\Omega)$ for all $i\in \{1,\dots,n\}$. Let us decompose 
\[
\partial_{ii} u = (\partial_{ii} u)_+ - (\partial_{ii} u)_-. 
\] 
Then, we have
\[
(\partial_{ii} u)_- \in L^1_{\rm loc}(\Omega)\qquad\text{for all}\quad i\in \{1,\dots,n\},
\]
that is, the negative part of  $\partial_{ii} u$ is not singular.
\end{prop}

\begin{remark}
The assumption that $\partial_{ii} u\in \mathcal{M}(\Omega)$ is necessary. In general, it is not true that subharmonic functions have second derivatives (in the distributional sense) that have finite mass. As an example, take $u(x) = \log(|x|)$ for $x\in \R^2$. 
\end{remark}
\begin{remark}
We will see in Theorem~\ref{Countp} below that the improvement of regularity from $\mathcal{M}$ to $L^1$ for $(\partial_{ii} u)_-$ is optimal within $L^p$ spaces.
\end{remark}

In order to show Proposition \ref{deri}, let us recall the main result of \cite{GUIANN}. Let $\mathcal{A}$ be a linear operator of order $k$ acting on vector valued functions $\varphi \in C^\infty(\R^m,\R^n)$
\[
\mathcal{A}(\varphi) \in C^\infty(\R^m,\R^N),\qquad \mathcal{A} (\varphi) := \sum_{| \alpha |\leq k} A_\alpha \partial^{\alpha} \varphi,\quad A_\alpha \in \R^{N\times n}, 
\]
where $\alpha$ is a multi-index of length $|\alpha|$. We can then give meaning to expressions of the form 
$
\mathcal{A}(\mu) = \nu,
$
 simply by passing to the weak form of the latter. Let us define, 
\[
\mathbb{A}_k(\xi) = \sum_{|\alpha| = k}A_\alpha\xi^\alpha\quad\text{for each $\xi\in \R^m$},\qquad \Lambda_{\mathcal{A}} := \bigcup_{\xi \neq 0} \Ker(\mathbb{A}_k(\xi)) \subset \R^n.
\]
Here, $\xi^\alpha = \xi_{\alpha_1}\dots\xi_{\alpha_k}$ if $\alpha = (\alpha_1,\dots,\alpha_k)$ and $\Lambda_{\mathcal{A}}$ is called the $\Lambda$-cone associated to $\mathcal{A}$. Then \cite[Theorem 1.1]{GUIANN} tells us that, if $\mu$ is such that 
\[
\mu = Fdx + Pd\mu^s,
\]
for some positive and singular $\mu^s$, and it solves $\mathcal{A}(\mu) = \nu$ for some measure $\nu$ and operator $\mathcal{A}$, we have
\[
P(x) \in \Lambda_{\mathcal{A}},\qquad \text{ at }\mu^s \text{-a.e. }x.
\]

\begin{proof}[Proof of Proposition \ref{deri}]

Consider the decomposition
\[
\mu_{ii} := \partial_{ii} u = f_{ii}dx + d\mu_{ii}^s,
\]
with $\mu_{ii}^s$ finite and singular with respect to the Lebesgue measure, and let us define the vector-valued measure $\mu := (\mu_{11},\dots,\mu_{nn})$. Then, if $u$ is subharmonic, $
\sum_{i}\mu_{ii}\ge 0$.
We want to show $\mu_{ii}^s \ge 0$ for all $i\in \{1,\dots,n\}$. Let us define $\mathcal{A}$ acting on $C^\infty(\R^n,\R^n)$ and taking values in $\R^{n\times n}$ by
\[
\mathcal{A}(v)_{ij} := \partial_{ii}v_j - \partial_{jj}v_i,\quad \text{for all}\quad  i,j\in \{1,\dots,n\},  \quad \text{for all}\quad v \in C^\infty(\R^n,\R^n).
\]
Observe that, with this definition, we have $\mathcal{A}(\mu) = 0$ in the sense of distributions.

From \cite[Theorem 1.1]{GUIANN}, the polar vector $P(x)$ of its singular part belongs to $\Lambda_{\mathcal{A}}$ at $\mu^s := \sum_{i}\mu_{ii}^s$ -a.e. $x$. Notice that $\Lambda_{\mathcal{A}}$ is given by
\[
\Lambda_{\mathcal{A}} = \{v \in \R^n: \text{either $v_i \ge 0$ for all $i$ or $v_i \le 0$ for all $i$}\},
\]
since by definition $v \in \Lambda_{\mathcal{A}}$ if and only if there exists $\xi \in \R^m\setminus\{0\}$ such that \[
\xi_i^2v_j = \xi_j^2v_i\qquad\text{for all $i, j\in \{1,\dots,n\}$}.
\] 

Hence, either $P_i(x) \ge 0$ for all $i$ or $P_i(x) \le 0$ for all $i$, at $\mu^s$-a.e. $x\in \R^n$. Since $\Delta u \ge 0$, we get the desired result, $P_i(x) \ge 0$ for all $i$ at $\mu^s$-a.e. $x\in \R^n$.
\end{proof}

\section{Counterexample for Lipschitz subharmonic functions}

The main results of this section are the following two theorems. For the first theorem below, we observe that requiring a subharmonic function to be Lipschitz is not enough to guarantee the existence of second derivatives as measures. We refer to Section~\ref{sec.obst} for an interpretation in the obstacle problem case. 

\begin{theorem}\label{Count}
There exists $u\in {\rm Lip}(B_1)$ subharmonic in the distributional sense such that $\partial_{11}u$ is not a finite measure in any open subset $\Omega' \subset B_1$. 
\end{theorem}
\begin{proof}
In \cite[Theorem 3]{CFM} it is proved that there exists a separately convex (in particular, subharmonic) and Lipschitz function $g: \R^2 \to \R$ such that $\partial_{12}g$ is not a measure and $g(z) = |z|^2$ for $z \notin (0,1)^2$. Here, separately convex means that $x\mapsto g(x,y)$ is convex for every $y$ and $y \mapsto g(x,y)$ is convex for every $x$. To find a counterexample to our statement, simply take
\[
u(x,y) := g(x+y,x-y).
\]
See \cite[Remark 1]{CFM} to find such a $u$ with degenerate behaviour in every open subset.
\end{proof}

Our second result states that, even if $\partial_{11} u$ and $\partial_{22} u$ are finite measures, in general the improvement of regularity obtained in Proposition~\ref{deri} is optimal. 

\begin{theorem}\label{Countp}
There exists $u\in {\rm Lip}(B_1)$ subharmonic in the distributional sense,  such that $u\in W^{2,1}(B_1)$ but
\[
(\partial_{ii}u(x))_{-}\notin L^q(\Omega')\quad\text{for any }\quad q > 1,\quad\text{for}\quad i = 1,2,
\]
for any open subset $\Omega' \subset B_1$.
\end{theorem}

The proof is an adaptation of the convex integration methods developed in \cite{CFM} to produce counterexamples to $L^1$ estimates (cf. \cite[Lemmas 1 and 2]{CFM}). First of all, we need to recall some of the main points of the convex integration methods of \cites{SMVS,KIRK}. Here we will only describe the method for symmetric matrices in $\R^{2\times 2}$, but the theory is well developed for vectorial problems in $\R^{n\times m}$, and more generally for general linear operators, see \cite{ST}.

\subsection{Laminates of finite order and elementary splitting}

We say that $A,B \in \Sym(2)$ are \emph{rank-one connected} if
\[
\rank(A-B) = 1.
\]
We have the following, see \cite[Proposition 3.4]{KIRK}:
\begin{lemma}[\cite{KIRK}]\label{sl}
Let $A,B,C \in \Sym(2)$, with $\rank(B - C) = 1$, and $A = tB + (1 - t)C$, for some $t \in [0,1]$. Let also $\Omega \subset \R^2$ be a fixed open domain. Then, for every $\varepsilon > 0$, one can find a Lipschitz piecewise affine map $f_\eps: \Omega\to \R^2$ such that
\begin{enumerate}
\item $f_\eps(x) = Ax$ on $\partial\Omega$ and $\|f_\eps - A\|_{\infty} \le \varepsilon$;
\item $D f_\eps(x) \in \Sym(2)\cap B_{\varepsilon}([B,C])$ (here $[B,C]$ denotes the segment connecting $B$ and $C$);
\item $|\{x \in \Omega: D f_\eps(x) = B\}|\ge (1 - \varepsilon)t|\Omega|$ and $|\{x \in \Omega: D f_\eps(x) = C\}|\ge (1 - \varepsilon)(1 - t)|\Omega|$.
\end{enumerate}
Moreover, for every continuous $\Phi \in C(\R^{2\times 2})$, it holds
\[
\frac{1}{|\Omega|}\int_{\Omega}\Phi(Df_\eps)dx \to \int_{\R^{2\times 2}}\Phi(X) d\nu(X).
\]
\end{lemma}
Denote with $\mathcal{P}(U)$ the space of probability measures with support in $U \subset \Sym(2)$.
\begin{definition}
Let $\nu,\mu \in \mathcal{P}(U)$, $U \subset \Sym(2)$ open. Let $\nu = \sum_{i= 1}^N\lambda_i\delta_{A_i}$. We say that $\mu$ can be obtained via \emph{elementary splitting from }$\nu$ if for some $i \in \{1,\dots,N\}$, there exist $B,C \in U$, $\lambda \in [0,1]$ such that
\[
\rank(B-C) = 1,\quad [B,C] \subset U, \quad A_i = sB + (1-s)C,
\]
for some $s \in (0,1)$ and
\[
\mu = \nu +\lambda\lambda_i(-\delta_{A_i} + s\delta_B + (1-s)\delta_C).
\]
A measure $\nu = \sum_{i= 1}^r\lambda_i\delta_{A_i}\in \mathcal{P}(U)$ is called a \emph{laminate of finite order} if there exists a finite number of measures $\nu_1,\dots,\nu_r \in \mathcal{P}(U)$ such that
\[
\nu_1 = \delta_X,\quad \nu_r = \nu
\]
and $\nu_{j + 1}$ can be obtained via elementary splitting from $\nu_j$, for every $j\in \{1,\dots,N-1\}$. For a laminate of finite order $\nu$, we define its \emph{trail} to be the union of all couples of matrices $\{B,C\}$ as above. This finite set is denoted with $\tra(\nu)$.
\end{definition}

Iterating Lemma \ref{sl} and using the definition of elementary splitting, see for instance \cite[Lemma 3.2]{SMVS} for an analogous result, one can prove the following:
\begin{lemma}\label{ind}
Let $\Omega \subset \R^2$ be an open domain. Let $U \subset \Sym(2)$ be an open set and let $\nu = \sum_{i = 1}^r\lambda_i\delta_{A_i} \in \mathcal{P}(U)$ be a laminate of finite order with barycenter $A \in \Sym(2)$, i.e.: 
\[
A = \int_{\R^{2\times 2}} Xd\nu(X).
\]
Then, for any $b \in \R^2$ and $\varepsilon>0$, the map $f_0(x):= Ax + b$ admits on $\Omega$ an approximation of piecewise affine, equi-Lipschitz maps $f_\eps \in W^{1,\infty}(\Omega,\R^2)$ with the following properties:
\begin{enumerate}
\item $f_\eps(x) = Ax +b$ on $\partial \Omega$ and $\|f_\eps - Ax - b\|_\infty \le \eps$;
\item $D f_\eps(x) \in \Sym(2)\cap \bigcup_{\{X,Y\} \in \tra(\nu)}B_{\varepsilon}([X,Y])$;
\item $|\{x \in \Omega: Df_\eps(x) = A_i\}| \ge (1-\eps)\lambda_i|\Omega|,\forall i$.
\end{enumerate}
Moreover, for every continuous $\Phi \in C(\R^{2\times 2})$, it holds
\begin{equation}\label{convv}
\frac{1}{|\Omega|}\int_{\Omega}\Phi(Df_\eps)dx \to \int_{\R^{2\times 2}}\Phi(X) d\nu(X).
\end{equation}
\end{lemma}

\subsection{The counterexample}

We let
\[
\Sym_+(2):= \{X \in \Sym(2): \tr(X) \ge 0\}
\]
and
\[
\diag_+(2) := \{X \in \Sym_+(2): x_{12}= x_{21} = 0\}.
\]

Let us start by performing the following construction of a sequence of laminates of finite order, that will be crucial to prove our desired result. 

\begin{lemma}\label{n}
Let $1 < p < \log_2(3)$ be fixed. For all $k \in \N$, there exists a laminate of finite order $\nu_{p,k}$ supported in the space $\diag_+(2)$ such that the following hold, for some universal constants $c_*, C_*>0$:
\begin{enumerate}
\item\label{11} $\displaystyle\int X d\nu_{p,k}(X) = k\id$;
\item\label{22} $\displaystyle\int\left\{|x_{11}| + |x_{22}| \right\} d\nu_{p,k}(X) = \displaystyle\int |X| d\nu_{p,k}(X) = C(p)k$ for some $0<C(p) \le C_* < \infty$;
\item\label{33} $\displaystyle\int (x_{ii})^q_- d\nu_{p,k}(X) = c_i(p, q)k^q$, for $i = 1,2$, for all $q\in [1, 2)$ and for some $c_i(p, q) \ge c_*>0$;
\item\label{44} $\displaystyle\tra(\nu_{p,k}) \subset \diag_+(2)$.
\item\label{55} $\nu_{p,k}$ contains a Dirac delta at $2k \id$ with weight $\frac{1}{2^p}$.
\end{enumerate}
\end{lemma}
\begin{proof}
Define the following probability measure $\nu_{p,k}$:
\[
\nu_{p,k} = \alpha \delta_{kA} + \beta(1-\alpha)\delta_{2k\id} + (1-\beta)(1-\alpha)\delta_{k B},
\]
where
\begin{equation}
\label{eq.abalfbet}
A := \left(
\begin{array}{cc}
\frac{2^p - 3}{2^p - 1} & 0\\
0 & 1
\end{array}
\right),\quad  B := \left(
\begin{array}{cc}
2 & 0\\
0 & -\frac{2}{2^p - 1}
\end{array}
\right),
\quad \alpha := \frac{2^p - 1}{2^p + 1}, \quad \beta := \frac{2^p + 1}{2^{p + 1}}.
\end{equation}
Define also
\[
M := \left(
\begin{array}{cc}
2 & 0\\
0 & 1
\end{array}
\right).
\]
Statements \eqref{11}-\eqref{22}-\eqref{33} can be checked by direct computation, while \eqref{55} is clearly true by definition. Notice that the assumption $p < \log_2(3)$ is only used to check \eqref{33} for $i = 1$. To see that $\nu_{p,k}$ is a laminate of finite order (see the splitting in Figure~\ref{fig.1}) and \eqref{44}, we consider the following construction for $\nu_{p,k}$. First, we split $\delta_{k\id}$ as
\[
\alpha \delta_{kA} + (1-\alpha)\delta_{kM}.
\]
This is an elementary splitting since $\det(A - M) = 0$. Then, we split $\delta_{kM}$ as
\[
\beta \delta_{2k\id} + (1-\beta) \delta_{kB},
\]
and again we have $\det(2\id - B) = 0$. The proof is finished.
\end{proof}	

\begin{figure}
\includegraphics[scale = 1]{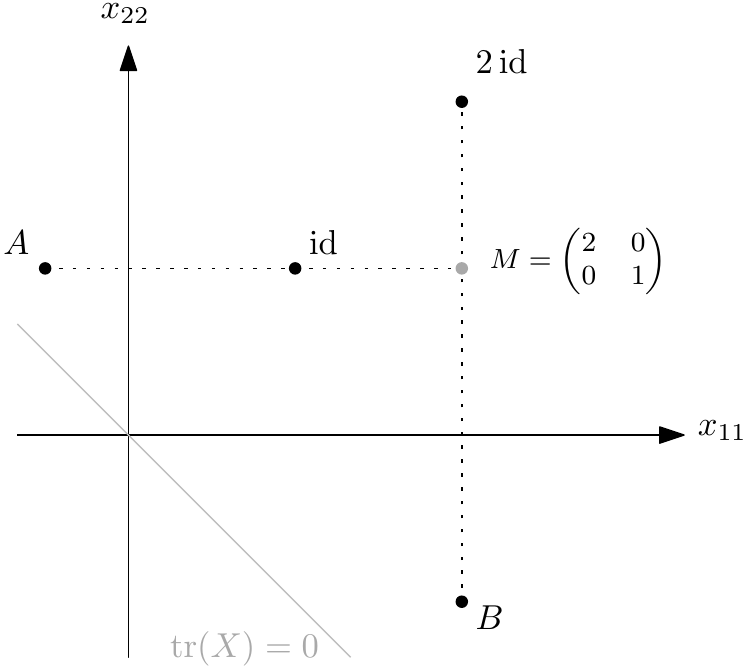}
\caption{The splitting in the laminate $\nu_{p, 1}$ from the construction in Lemma~\ref{n}.}
\label{fig.1}
\end{figure}


Before constructing the main counterexample, Theorem~\ref{Countp}, let us show first, for the sake of clarity, the following weaker version of our result.
In particular, we show here how the previous construction, combined with Lemma~\ref{ind}, yields that we can find subharmonic functions $u$ with $\partial_{ii} u \in \mathcal{M}(B_1)$, and $\|(\partial_{ii} u)_-\|_{L^p(B_1)}$ arbitrarily large for $p$ arbitrarily close to 1. 

\begin{theorem}
\label{thm.mainweak}
For every $p > 1$ and $j \in \N$, there exists a function $u_{p,j} :\Omega := (0,1)^2 \to \R$ such that
\begin{enumerate}[\quad(a)]
\item\label{aa} $u_{p,j} = \frac{|x|^2}{2}$ and $Du_{p,j} = x$ on $\partial \Omega$;
\item\label{bb} $\|u_{p,j}(x)-  \frac{|x|^2}{2}\|_{W^{1,\infty}(\Omega)}\le \frac{1}{j}$;
\item\label{cc} $\Delta u_{p,j} \ge 0$ in the sense of distributions;
\item\label{dd} $\|D^2u_{p,j}\|_{L^1}\le C$ for all $j$, for some $C > 0$ universal;
\item\label{ee} $\int_{\Omega} (\partial_{ii}u_{p,j})_{-}^p dx \ge j, \quad \text{for}\quad i = 1,2$.
\end{enumerate}
\end{theorem}

The proof of the previous theorem is a direct consequence of the following intermediate lemma, together with Lemma~\ref{ind}.

\begin{lemma}\label{fundlem}
Let $1 < p < \log_2(3)$ be fixed. There exists a sequence of laminates of finite order $(\nu_p^{(m)})_{m\in \N}$ supported in the space $\diag_+(2)$ such that the following hold:
\begin{enumerate}
\item\label{1} $\int X d\nu_p^{(m)}(X) = \id$, $\forall m \in \N$ and hence $\int |\tr(X)| d\nu_p^{(m)} = 2, \forall m \in \N$;
\item\label{2} $\sup_m\int \left\{|x_{11}| + |x_{22}| \right\} d\nu_p^{(m)}(X) < +\infty$;
\item\label{3} $\int (x_{ii})_{-}^qd\nu_p^{(m)}(X) \to + \infty$  as $m\to +\infty$, for $i = 1,2$, for all $q\ge p$; 
\item\label{4} $\tra(\nu_p^{(m)}) \subset \diag_+(2)$.
\end{enumerate}
\end{lemma}
\begin{proof}
We apply Lemma \ref{n} above. We start with $\nu_p^{(0)} := \delta_{\id}$. Subsequently, we take $\nu_p^{(1)} := \nu_{p,1}$, the latter being the one constructed in Lemma \ref{n} with $k = 1$. Now this laminate contains a Dirac's delta at $2\id$ with weight $\lambda = \frac{1}{2^p}$. We split this using $\nu_{p, 2}$, constructed in Lemma \ref{n}:
\[
\nu^{(2)}_p := \nu^{(1)}_p - \lambda\delta_{2\id} + \lambda\nu_{p,2}.
\]
We continue iteratively, and define
\[
\nu^{(m)}_p := \nu^{(m-1)}_p - \lambda^{m - 1}\delta_{2^{m - 1}\id} + \lambda^{m - 1}\nu_{p,2^{m - 1}}.
\]
Now, \eqref{1} and \eqref{4} hold by construction. Let us show \eqref{2}. Let $a_m := \int\{|x_{11}| + |x_{22}|\}d\nu_p^{(m)}$ and observe
\[
a_m  = a_{m-1} - \lambda^{m - 1}2^{m}+ \lambda^{m - 1}\int (|x_{11}| + |x_{22}|)d\nu_{p,2^{m - 1}}= a_{m-1}- (2-C)\lambda^{m-1}2^{m-1},
\]
where we also use \eqref{22} of Lemma \ref{n}. It follows that
\[
|a_m - a_{m - 1}| = |C - 2||\lambda^{m - 1}2^{m - 1}| \overset{\lambda = 2^{-p}}{=} |C- 2|\frac{2^{m - 1}}{2^{(m - 1)p}} = |C- 2|2^{(1-p)(m - 1)}.
\]
Since $p>1$, $\sum_{m}|a_m - a_{m - 1}| < + \infty$ and \eqref{2} holds. Let us finally turn to \eqref{3}. As before, let $b_{m,i} := \int (x_{ii})^q_{-}d\nu^{(m)}$. Use the definition of $\nu^{(m)}$ to find:
\[
b_m = b_{m - 1} + \lambda^{m - 1}\int (x_{ii})_{-}^qd\nu_{p,2^{m - 1}}.
\]
By \eqref{33} of Lemma \ref{n}, we see that the latter is equivalent to
\[
b_m - b_{m - 1} = c_i\lambda^{m - 1}2^{q(m - 1)} \overset{\lambda = 2^{-p}}{=} 2^{(q-p)(m - 1)}c_i.
\]
Since $c_i > 0$ for $i = 1,2$ and $q \ge p$, \eqref{3} readily follows.
\end{proof}

And as a direct application of the previous lemma combined with Lemma \ref{ind} we get Theorem~\ref{thm.mainweak}.

\begin{proof}[Proof of Theorem~\ref{thm.mainweak}]
Let $j\in \N$ and $p>1$ be fixed. Consider $\nu^{(m)}_{p}$ of Lemma \ref{fundlem} with $m$ large enough to guarantee that
\[
\int (x_{ii})_{-}^pd\nu_p^{(m)}(X) > j, \quad \text{for }\quad  i = 1,2.
\]
By Lemma \ref{ind}, for an $\eps$ to be fixed later we find a map $f_\eps : \Omega \to \R^2$ such that
\begin{enumerate}
\item\label{a} $f_\eps(x) = x$ on $\partial \Omega$ and $\|f_\eps-\id\|_{\infty} \le \eps$;
\item\label{b} $Df_{\eps}(x) \in \Sym(2) \cap \bigcup_{\{X,Y\} \in \tra(\nu_{p}^{(m)})}(B_\eps([X,Y]))$
\end{enumerate}
In particular, $f_{\eps} = \nabla u_{\eps}$ for some $u_{\eps}: \Omega \to \R$. We choose $u_{\eps}(x)$ in such a way that $u_{\eps}(x) = \frac{|x|^2}{2}$ on $\partial \Omega$, so that \eqref{aa}-\eqref{bb} are satisfied, provided we choose $\eps = \eps(p, j)$ sufficiently small. By construction of $\nu^{(m)}_p$ in Lemma \ref{fundlem}, we also see that $\tra(\nu^{(m)}_p)$ is compactly contained in the open set $\{X \in \diag(2): x_{11} + x_{22} > 0\}$. Therefore, if $\eps$ is chosen sufficiently small, we also have \eqref{cc}. Finally, \eqref{convv} yields:
\[
\frac{1}{|\Omega|}\int_{\Omega}g(Df_\eps(x))dx \to \int_{\R^{2\times 2}}g(X)d\nu_p^{(m)}, \quad \text{ as $\eps \to 0$}.
\]
Use the latter with $g = \sum_{i,j} |x_{ij}|$ and $g = (x_{ii})^p_{-}$ and use properties \eqref{2}-\eqref{3} of $\nu_p^{(m)}$ to conclude the validity of \eqref{dd}-\eqref{ee} for all $u_\eps$ with $\eps = \eps(p, j)$ very small. Take any such $\eps$ and denote this function with $u_{p,j}$. This concludes the proof.
\end{proof}

Let us now use the previous ideas to construct the actual counterexample, Theorem \ref{Countp}. We do so by combining Lemma \ref{n} and Lemma \ref{ind}, with some of the ideas in the proof of Lemma~\ref{fundlem}.

\begin{proof}[Proof of Theorem \ref{Countp}]
Let us consider a sequence $(p_j)_{j\in \N}$ with $p_j \downarrow 1$ to be chosen, and let us denote by $A_j, B_j\in \R^{2\times 2}$, $\alpha_j, \beta_j > 0$, the parameters defined in the proof of Lemma~\ref{n}, \eqref{eq.abalfbet}, with $p = p_j$. 

Let us consider the laminate 
\[
\mu_{j, k} := \alpha_j \delta_{k A_j} + \beta_j (1-\alpha_j)\delta_{2k\id} + (1-\beta_j)(1-\alpha_j)\delta_{k B_j}
\]
as in the construction in Lemma \ref{n}.

We start by defining $f_0(x) := x$ on $\Omega= (0,1)^2$ and considering the laminate $\mu_{1, 1}$. We take now a compactly supported $\Omega_1 \subset \Omega$, $|\Omega\setminus \Omega_1| \le \eps_1$  and use Lemma \ref{ind}, to find a map $g_1$ with
\begin{enumerate}[\quad(i)]
\item\label{A} $g_1(x) = x$ on $\partial \Omega_1$ and $\|g_1-\id\|_{L^\infty(\Omega_1)} \le \eps_1$;
\item\label{B} $Dg_1(x) \in \Sym(2) \cap \bigcup_{\{X,Y\} \in \tra(\mu_{1,1})}B_{\eps_1}([X,Y]) \subset \{X\in \Sym(2): \tr(X) > 0\}$;
\item\label{C} $|\{x \in \Omega: Dg_{1}(x) = A_1\}| \ge (1-\eps_1)\alpha_1|\Omega'|$;
\item\label{D} $|\{x \in \Omega: Dg_{1}(x) = 2\id\}| \ge (1-\eps_1)\beta_1(1-\alpha_1)|\Omega'|$;
\item\label{E} $|\{x \in \Omega: Dg_{1}(x) = B_1\}| \ge (1-\eps_1)(1-\beta_1)(1-\alpha_1)|\Omega'|$,
\end{enumerate}
for some $\eps_1$ small enough to be chosen later. By \eqref{B}, we find that $g_1 = \nabla u_1$ for some $u_1$ satisfying $u_1 = \frac{|x|^2}{2}$ on $\partial \Omega_1$. Furthermore by \eqref{convv} and the properties of $\mu_{1,1}$, we additionally have, if $\eps_1$ is chosen sufficiently small,
\[
\int_{\Omega_1}|D^2u_1| \le 2C_*|\Omega_1|,
\]
and
\[
\int_{\Omega_1}(\partial_{ii}u_1)_{-}^qdx \ge \frac{c_*}{2}|\Omega_1|, \quad \text{for}\quad i = 1,2,\quad \text{for any}\quad q\in [1, 2)
\]
where $C_*,c_*$ are the constants appearing in Lemma \ref{n}. We therefore define $f_1$ as $x$ on $\Omega\setminus \Omega_1$ and $f_1 = g_1$ on $\Omega_1$. By construction, since $g_1$ is piecewise affine, the set $\Omega_2:= \{Dg_1 = 2\id\}$ is essentially open. On this set, we substitute $g_1$ with $g_2$, constructed using Lemma \ref{ind} with the laminate $\mu_{2, 2}$. On a connected component of $\Omega_2$, we have that $g_1(x) = 2x + b$, for some constant $b$. We choose as boundary datum for $g_2(x)$ exactly $2x + b$, in such a way that this replacement is still Lipschitz. Furthermore, up to subdiving them, we can suppose that the diameter of any connected component of $\Omega_2$ does not exceed $\eps_2$. With this choice, we see that
\[
\|g_2 - g_1\|_{L^\infty(\Omega_2)} \le \eps_2\|Dg_2 - Dg_1\|_{L^\infty(\Omega_2)} \le 2^2\eps_2. 
\]
The analogous properties to the ones above hold for $g_2$ and $\eps_2$. We define $f_2$ as $f_1$ on $\Omega \setminus \Omega_2$ and $g_2$ on $\Omega_2$. We now iterate this procedure on the open set in which $\{Dg_2 = 4\id\}$ with the laminate $\mu_{3, 4}$. Inductively, we will be given nested open sets 
\[
\Omega_{j}\subset \Omega_{j-1}, \quad \Omega_j := \{x \in \Omega_{j - 1}: Dg_{j - 1} = 2^j\id\},
\]
and sequences of Lipschitz maps $g_j$, $f_j$, and functions $u_j$ such that $\nabla u_j = f_j$ with the following properties
\begin{enumerate}[\quad (a)]
\item\label{AA} $f_j = f_{j - 1}$ in $\Omega\setminus \Omega_{j}$ and $f_j = g_j$ in $\Omega_{j}$;
\item\label{BB} $\|g_j - g_{j -1}\|_{L^\infty(\Omega_j)} \le 2^j\eps_j$;
\item\label{CC} $Dg_j \in \Sym(2)\cap \{X: \tr(X) > 0\}$, for all $j\in \N$, almost everywhere;
\item\label{DD} $|\{x \in \Omega_j: Dg_{j}(x) = 2^{j}A_j\}| \ge (1-\eps_j)\alpha_j|\Omega_{j - 1}|$;
\item\label{EE} $|\{x \in \Omega_j: Dg_{j}(x) = 2^{j+ 1}\id\}| = |\Omega_{j}|\ge (1-\eps_j)\beta_j(1-\alpha_j)|\Omega_{j - 1}|$;
\item\label{FF} $|\{x \in \Omega_j: Dg_{j}(x) = 2^jB_j\}| \ge (1-\eps_j)(1-\beta_j)(1-\alpha_j)|\Omega_{j - 1}|$;
\item\label{GG} $\displaystyle\int_{\Omega_j}|D^2u_j| \le 2C_* 2^{j} |\Omega_j|$;
\item\label{HH} $\displaystyle\int_{\Omega_j\setminus\Omega_{j+1} }(\partial_{ii}u_j)_{-}^{q} \ge cc_* 2^{qj}|\Omega_j\setminus \Omega_{j+1}|, $ for $i = 1, 2$, and for any $q \in [1, 2)$,
\end{enumerate}
where properties \eqref{GG}-\eqref{HH} follow by applying \eqref{convv} and Lemma~\ref{n} with $\eps_j$ small enough. 

We choose $\eps_j \le \frac{1}{4^j}$ for all $ j$. By \eqref{AA} and \eqref{BB}, we see then that 
\[
\|f_j - f_{j - 1}\|_{L^\infty(\Omega)} \le \frac{1}{2^j}.
\]
This implies that $f_j$ converge strongly in $L^\infty$ to $f \in L^\infty$. This also implies that $u_j$ converge strongly in $W^{1,\infty}$ to a function $u\in W^{1,\infty}(\Omega)$. Furthermore, $D^2u_j$ converge weakly-$*$ in the sense of measures to $D^2u$. By \eqref{CC}, this function $u$ will also be subharmonic. Moreover, by construction $D^2 u$ enjoys the following property:
\[
D^2 u = D f_j = D^2 u_j\qquad\text{in}\quad \Omega_j \setminus \Omega_{j+1},\qquad\text{for all}\quad j\in \N.
\]
Hence, if we assume that $D^2 u\in L^1_{\rm loc}$ (which is shown at the end) then, since $|\Omega_j|\downarrow 0$ as $j\to \infty$,
\begin{equation}
\label{eq.plug1}
\int_\Omega |D^2 u| \le \sum_{j\ge 1} \int_{\Omega_j} |D^2 u_j| \le 2C_*\sum_{j \ge 1} 2^j |\Omega_j|,
\end{equation}
and 
\begin{equation}
\label{eq.plug2}
\int_{\Omega_1} (\partial_{ii} u)_-^{q} \ge \sum_{j\ge 1} \int_{\Omega_j\setminus\Omega_{j+1}} (\partial_{ii} u_j)^q_-\ge \frac{c_*}{2} \sum_{j\ge 1} 2^{qj}(|\Omega_j| - |\Omega_{j+1}|) \ge  c'\sum_{j\ge 1}2^{qj}|\Omega_j|,
\end{equation}
where we are using properties \eqref{GG}-\eqref{HH}. 
We can employ property \eqref{EE} to write:
\[
\begin{split}
|\Omega_j| \ge (1-\eps_j) \beta_j (1-\alpha_j) |\Omega_{j-1}|& \ge c \prod_{m = 1}^j (1-\eps_m) \beta_m (1-\alpha_m)   \\
& \ge c \prod_{m = 1}^j (1-\eps_m) 2^{-p_m} \ge \frac{c}{2} \prod_{m = 1}^j  2^{-p_m} ,
\end{split}
\]
where in the last inequality we are taking $\eps_m$ small enough, and we are also using the definitions of $\alpha_j$ and $\beta_j$, so that $\beta_j(1-\alpha_j) = 2^{-p_j}$. On the other hand, using properties \eqref{DD} and \eqref{FF}, 
\[
\begin{split}
|\Omega_j|& \le \left(1-(1-\eps_j)[\alpha_j + (1-\beta_j)(1-\alpha_j)]\right)|\Omega_{j-1}|\\
& \le \left(\beta_j(1-\alpha_j) +\eps_j(1-\beta_j(1-\alpha_j))\right)|\Omega_{j-1}|\\
& \le \left(2^{-p_j} + \eps_j\right)|\Omega_{j-1}|\le C\prod_{m = 1}^j \left(2^{-p_m} + \eps_m\right)\le 2C\prod_{m = 1}^j 2^{-p_m} 
\end{split}
\]
again, for $\eps_m$ small enough (now, depending on $p_m$). 

Plugging these estimates into \eqref{eq.plug1}-\eqref{eq.plug2}, we get 
\begin{equation}
\label{eq.first}
\int_\Omega |D^2 u| \le C \sum_{j \ge 1}2^{\sum_{m = 1}^j (1-p_m)}
\end{equation}
and 
\[
\int_{\Omega_1} (\partial_{ii} u)_-^{q} \ge c \sum_{j \ge 1}2^{\sum_{m = 1}^j (q-p_m)}.
\]
Notice that, for this second integral, since $p_m \downarrow 1$ and $q > 1$, the sum always diverges, so $(\partial_{ii} u)_- \notin L^q(\Omega)$ for any $q > 1$. 

Regarding \eqref{eq.first}, take $p_m = 1+\kappa \frac{1}{m}$, with $\kappa = \frac{2}{\log(2)}$. Then, 
\[
\sum_{m = 1}^j (1-p_m) = -\kappa \sum_{m = 1}^j \frac{1}{m} < -\kappa \log(j+1) = -2\log_2(j+1), 
\]
and so 
\[
\int_\Omega |D^2 u| \le C \sum_{j \ge 1}2^{\sum_{m = 1}^j (1-p_m)} < C \sum_{j \ge 1} 2^{-2\log_2(j+1)} = C\sum_{j\ge 1} \frac{1}{(j+1)^2} < +\infty. 
\]

For $u$ to be our counterexample, it only remains to show that, with this construction, $D^2 u\in L^1_{\rm loc}$. Equivalently, we show that
\[
\int_{\Omega_\infty} |D^2 u| \le \lim_{j\to \infty}\int_{\Omega_j}|D^2 u| = 0,\qquad\text{where}\quad \Omega_\infty := \bigcap_{j\ge 1}\Omega_j.
\]

Indeed, by the lower semi-continuity of the total variation on open sets, 
\[
\int_{\Omega_j}|D^2 u| \le \lim_{\ell \to \infty} \int_{\Omega_j}|D^2 u_\ell| \le \sum_{\ell \ge j}\int_{\Omega_\ell}|D^2 u_\ell|\to 0,\quad\text{as}\quad j \to \infty,
\]
by the previous reasoning (since with our choice, the series in \eqref{eq.plug1} converges), which gives the desired result. \end{proof}

Combining the previous results we now get the proof of Theorem~\ref{thm.main}. 

\begin{proof}[Proof of Theorem~\ref{thm.main}]
It is a consequence of Proposition~\ref{deri} and Theorem~\ref{Countp}. 
\end{proof}

\section{An application to the obstacle problem}
\label{sec.obst}

The obstacle problem is the following constrained minimization problem:

 Let $g\in W^{1,2}(B_1)$ be the boundary datum, and $\phi\in W^{1,2}(B_1)$ the obstacle. Then, we consider the minimizer of 
\begin{equation}
\label{eq.obstpb}
\min_{u \in \mathcal{A}(\phi)} \int_{B_1} |\nabla u|^2
\end{equation}
where 
\[
\mathcal{A}(\phi) := \{w \in W^{1,2}(B_1) : w - g\in W^{1,2}_0(B_1)\quad\text{and}\quad u \ge \phi~~\text{in}~~B_1\}. 
\]

The obstacle problem is a free boundary problem with applications to models in multiple areas, from physics to biology, economy, control theory, etc. We refer to \cites{DL76, Caf77, FS19, FR21} and references therein. The Euler--Lagrange equations of the previous problem are
\[
\left\{
\begin{array}{rcll}
u & \ge & \phi & \quad\text{in}\quad B_1\\
\Delta u & \le & 0& \quad\text{in}\quad B_1\\
\Delta u & = & 0 & \quad\text{in}\quad\{u > \phi\}.
\end{array}
\right.
\]
Alternatively, the solution can be obtained as the minimum of supersolutions that are above the obstacle and the boundary datum.

The first question about minimizers of the previous functional is that of regularity of solutions. Whenever the obstacle is smooth enough (say, $\phi\in C^{2}$) solutions to the obstacle problem are $C^{1,1}$ and not better in general.

Lowering the regularity of the obstacle also lowers the regularity of solutions. For example, whenever $\phi\in {\rm Lip}(B_1)$ then we expect the solution to be, at most, Lipschitz. 

The \emph{thin obstacle problem} is probably the most ubiquitous example of an obstacle problem that is merely Lipschitz, where one assumes that $\phi$ degenerates in some direction (it is equivalent to a Lipschitz obstacle, in fact); see \cite{PSU12, CSS08, CSV20, Fer21} for more details. For the thin obstacle problem, one imposes that $u \ge \phi$ only on a lower dimensional space, say $\{x_n  = 0\}$. However,  this is equivalent to slightly expanding the obstacle outside of $\{x_n = 0\}$ and considering a Lipschitz obstacle problem.

As expected, solutions to the thin obstacle problem are at most Lipschitz. However, the Lipschitz discontinuity is only observed, for smooth obstacles, across $\{x_n = 0\}$ and always in the same direction (since it is a superharmonic function). In fact, for the thin obstacle problem, we realize that solutions can always be touched by $C^{1}$ functions from above due to the asymmetry of the problem (the fact that $u \ge \phi$). So, in a way, solutions to the thin obstacle problem are more regular from above than from below. 

Even more, these solutions are $C^{1,1/2}$ on $\{x_n = 0\}$, which means that second derivatives can blow-up (roughly) at most like a power $-1/2$ on the thin space (at free boundary points). This, combined with the fact that the solution is semi-concave in the directions perpendicular to the thin space, implies that the positive part of second derivatives actually belongs to some $L^p$ for $p > 1$. That is, the positive part of second derivatives has higher regularity than the negative part, which is at most a measure. 

This observation for the case of thin obstacle problems suggests that this behaviour could potentially happen with any other Lipschitz obstacle (given the asymmetry of the problem), thus giving rise to an open problem in the field. The results in the previous sections show that, in general, such one-sided regularity is not true:

\begin{prop}
\label{prop.lip}
There exists a solution $u$ to the obstacle problem \eqref{eq.obstpb} with Lipschitz obstacle such that $u\in  W^{2,1}$ but $(D^2 u)_+ \notin L^p(B_1)$ for any $p  >1$.
\end{prop}
\begin{proof}
Take $v$ to be the function from Theorem~\ref{Countp}. Then $u = -v$ is both the obstacle and the solution to the obstacle problem, with Lipschitz obstacle, but $(D^2 u)_+ \notin L^p(B_1)$ for any $p > 1$. 
\end{proof}

That is, being a solution to the obstacle problem does not improve the regularity of the positive part of second derivatives, at  least in $L^p$ spaces. 

The case $p = 1$ remains an open problem. More precisely, by Theorem~\ref{Count} it is not true in general that solutions to the obstacle problem with Lipschitz obstacles always have second derivatives which are (signed) measures (again taking the solution constructed in Theorem~\ref{Count} as both the obstacle and the minimizer). Then, one needs to understand what are the conditions to be imposed on the obstacle to ensure that second derivatives are measures, so that we are able to apply Proposition~\ref{deri} above to get an improvement of regularity result.

This is an interesting problem that raises the following question: is it true that if $\phi$ is such that $D^2 \phi\in \mathcal{M}(B_1)$, then the solution to the obstacle problem $u$ also satisfies $D^2 u\in \mathcal{M}(B_1)$?  Observe that that $D^2 u$ is harmonic in $\{u > \phi\}$, and $D^2 u \ge D^2 \phi$ otherwise.











\section*{Acknowledgements}

We are thankful to Prof. H. Chang-Lara and Prof. X. Ros-Oton for showing us this problem in the context of the obstacle problem. We are also thankful to Prof. M. Colombo for the discussions on the topics of this paper.

This work has received funding from the SNF grant 200021\_182565.

\end{document}